\documentclass[11pt]{article}

\usepackage[utf8]{inputenc}
\usepackage[T1]{fontenc}
\usepackage{lmodern}

\usepackage{cite}

\usepackage{amsmath}
\usepackage{amsthm}
\usepackage{amssymb}
\usepackage{calrsfs}

\usepackage{fullpage}

\allowdisplaybreaks

\newtheorem{theorem}{Theorem}[section] 
\newtheorem{lemma}[theorem]{Lemma}
\newtheorem{corollary}[theorem]{Corollary}
\newtheorem{proposition}[theorem]{Proposition}

\theoremstyle{definition}
\newtheorem{example}[theorem]{Example}

\newtheorem{remark}[theorem]{Remark}

\newcommand{\de}{\partial}
\newcommand{\db}{\overline{\partial}}
\newcommand{\ddb}{\partial\overline{\partial}}

\newcommand{\Rb}{\mathbb{R}}
\newcommand{\Cb}{\mathbb{C}}
\newcommand{\Zb}{\mathbb{Z}}

\newcommand{\cL}{\mathcal{L}}
\newcommand{\cO}{\mathcal{O}}

\newcommand{\CP}{\Cb P}
\newcommand{\Ric}{\operatorname{Ric}}

\newcommand{\vol}{\operatorname{vol}}

\numberwithin{equation}{section}

\title{Steady Kähler-Ricci solitons on crepant resolutions of finite quotients of $\Cb^n$}

\author{Olivier Biquard and Heather Macbeth}

\date{}

\begin{document}

\maketitle

\begin{abstract} 
  We  prove the existence of steady Kähler-Ricci solitons on equivariant crepant resolutions of $\Cb^n/G$, where $G$ is a finite subgroup of $SU(n)$.
\end{abstract}

\renewcommand{\thefootnote}{}
\footnotetext{This material is based upon work supported by the National Science Foundation under Grant No.\ DMS-1440140 while the authors were in residence at the Mathematical Sciences Research Institute in Berkeley, California, during the Spring 2016 semester.}

\section{Introduction}

\subsection{Overview}

The `fixed points' of the Ricci flow are \emph{steady Ricci solitons}.  Such objects, natural generalizations of Ricci-flat metrics, are pairs $(g,X)$ of a Riemannian metric and a vector field, satisfying the elliptic partial differential equation
\[
\Ric(g)+\tfrac{1}{2}\cL_Xg=0.
\]
They arise in the study of the Ricci flow as  models of singularities \cite{DP07,GZ08,DS10,AIK15}, and as backward limits of ancient solutions.  They are also  critical points, in a suitable sense \cite{Has11}, of the Perelman $\mathcal{F}$-functional, and thus can be considered canonical among all such pairs $(g,X)$.

Steady Ricci solitons which are not Ricci-flat must be noncompact \cite{Ive93}.  The few known examples include several which are Kähler, with holomorphic vector field:  Hamilton's \emph{cigar soliton} \cite{Ham88} on $\Cb$, H.-D.\ Cao's generalizations \cite{Cao94} on $\Cb^n$ and $K_{\CP^{n-1}}$, and further generalizations by Dancer-M.\ Wang \cite{DW11} and B.\ Yang \cite{Yan12}.  Non-Kähler examples include the well-known constructions of Bryant and of Ivey.  All these examples -- and, we believe, all known examples -- are highly symmetric, and the soliton metric is given either explicitly or by solving an ODE.

In this article we use PDE methods to construct new steady Kähler-Ricci solitons $(M,\omega,X)$,  in all complex dimensions $n\geq 2$ (real dimensions $2n\geq 4$),  of infinitely many topological types in dimensions 2 and 3 at least.    Like all steady Kähler-Ricci solitons, they have first Chern class
\[
c_1(M)=[\Ric(\omega)]=[-\tfrac{1}{2}\cL_X\omega]=0.
\]
Our construction proceeds by taking Joyce's well-known family \cite{Joy00} of Ricci-flat Kähler metrics  on \emph{crepant resolutions} of orbifolds $\Cb^n/G$ (which automatically have $c_1(M)=0$), and modifying their metrics near infinity by \emph{gluing} them to a $G$-quotient of Cao's soliton on $\Cb^n$.   A precise statement is as follows:

\begin{theorem}\label{main}
  Let $G$ be a finite subgroup of $SU(n)$, which acts freely on $\Cb^n\setminus\{0\}$.  Let $M_G$ be a crepant resolution of $\Cb^n/G$, which is equivariant with respect to the action of $\Cb^*$.   Let $\mathfrak{k}$ be  a Kähler class on $M_G$ which contains an asymptotically locally Euclidean Kähler metric.
  
  Then for all $\epsilon$ sufficiently small, there exists a steady Kähler-Ricci soliton on  $M_G$ in the cohomology class $\epsilon\mathfrak{k}$, whose drift vector field is the extension to $M_G$ of the radial vector field $-2 r\partial/\partial r$ (this extension exists by the equivariance of the resolution).

  Moreover, outside a compact set these solitons have the form $\omega_0+\tfrac{i}{2}\de\db\Psi$, where $\omega_0$ is Cao's steady Kähler-Ricci soliton metric on $\Cb^n$ (descending to $\Cb^n/G$), and, for any $k$ and any $\lambda<n$,
  \[
|\nabla^k\Psi|_{\omega_0}  =O(e^{-\lambda t}),
\]
where $t$ is the distance to a fixed point in $M_G$.
\end{theorem}

Such gluing constructions have been performed before in other geometric settings. Early uses of the method include Taubes' construction \cite{Tau92} of anti-self-dual metrics, and Kapouleas' \cite{Kap90,Kap91} of  minimal surfaces.  More recently the method has been used \cite{Ngu09,Ngu13} to construct self-translators, the mean curvature flow analogue of steady Ricci solitons.  Joyce himself \cite{Joy00} used his Calabi-Yau's as building blocks for gluing constructions of manifolds of exceptional holonomy.

In the Kähler setting, Biquard-Minerbe \cite{BM11} used gluing methods to construct noncompact Calabi-Yau manifolds of complex dimension 2, of several different asymptotic behaviours.  Gluing techniques have also been used by various authors to construct constant scalar curvature and extremal Kähler metrics, see for example \cite{RS05,ArePac06,ArePacSin11,Sze12,BiqRol15}.

We also prove a general uniqueness statement.  We show that a steady K\"ahler-Ricci soliton is unique for its cohomology class, asymptotics (as previously mentioned, nontrivial steady solitons are noncompact), and drift vector field, in the following sense:
\begin{proposition}\label{prop:uniqueness}
  Let $(M,\omega,X)$ be a steady Kähler-Ricci soliton. Let $\omega+\frac i2\ddb\Psi$ be another steady Kähler-Ricci soliton for the same drift vector field $X$, such that at infinity we have
  \begin{enumerate}
  \item $\Psi\rightarrow0$;
  \item $X\Psi\rightarrow0$;
  \item $\ddb\Psi\rightarrow0$ (by comparison with $\omega$).
  \end{enumerate}
Then $\Psi=0$.
\end{proposition}

\subsection{Future directions}

Our result suggests numerous directions for future work.    First, the orbifolds $\Cb^n/G$ are only the simplest examples of Calabi-Yau cones.  Equivariant crepant resolutions of arbitrary Calabi-Yau cones are therefore good candidates for further manifolds admitting steady Kähler-Ricci solitons.  In generalizing the gluing procedure to these manifolds, the Calabi-Yau metrics of \cite{vC10,Got12} would replace the Joyce metrics used in this article as the gluing model near the singularity.  However, it is unclear what could replace the Cao soliton metric on $\Cb^n/G$ as the gluing model away from the singularity.

Furthermore, we conjecture the existence of a steady Kähler-Ricci soliton on the manifolds $M_G$, associated to the radial vector field $X$, in every Kähler class, not just small Kähler classes.  Producing such  Kähler-Ricci solitons amounts to solving a modified complex Monge-Ampère equation on the underlying complex manifold. 

In the compact case both continuity methods and variational methods have been used \cite{Zhu00,BerNys} for such modified complex Monge-Ampère equations, for example, the construction of compact shrinking K\"ahler-Ricci solitons.  The trick used by these authors is as follows:  the additional term of the Monge-Ampère equation in the soliton case is a derivative of the potential along the drift vector field; this is actually a moment map for a circle action; the image of the moment map does not depend on the particular Kähler (or symplectic) form, and so, since the manifolds are compact, it is a priori bounded.  In our noncompact setting the image of the moment map is unbounded so there is no a priori control.

In other recent work \cite{Sie13, CD16}, (noncompact) expanding Kähler-Ricci solitons have been constructed, solving the modified complex Monge-Ampère equation using a maximum principle which is not avalaible in the steady case.

New ideas  therefore seem necessary to solve the modified Monge-Ampère equation associated to steady K\"ahler-Ricci solitons.

\subsection{Outline}
In Sections \ref{radial-potential}--\ref{sec:analys-finite-quot}, we first review the construction of the Cao soliton, and then introduce the required analytical tools to handle the analysis in weighted Hölder spaces on the Cao soliton, with an exponential weight at infinity. Finally, for finite quotients of the Cao soliton, we combine this analysis with the weighted analysis at the singular point of the quotient.

In Section \ref{joyce-spaces} we review the material needed on crepant resolutions of $\Cb^n/G$, and on the canonical Kähler, Ricci-flat, asymptotically locally Euclidean metrics on these spaces.  We also discuss some specific examples of equivariant crepant resolutions.

The  proof of uniqueness (Proposition \ref{prop:uniqueness}) is given in Section  \ref{uniqueness-solitons}.

Finally, in Sections \ref{sec:appr-solut-monge}--\ref{sec:inverse-funct-theor}, we proceed to the gluing construction, using in particular a blowup analysis (Section \ref{first-order-sec}) to bound the inverse of the linearization of the problem in suitable weighted spaces.  Theorem \ref{main} is proved in Section \ref{sec:inverse-funct-theor}.

\section{$U(n)$-invariant Kähler metrics and the soliton equation} \label{radial-potential}

We begin by recalling Cao's construction \cite{Cao94} of a steady soliton on $\Cb^n$ and a one-parameter family of steady solitons on the bundle $\cO(-n)$ over $\CP^{n-1}$.  We habitually identify the complement of the zero section of $\cO(-n)$ with $(\Cb^n\setminus\{0\})/\Zb_n$, to which it is biholomorphic.

All Cao's solitons are invariant under the  natural action of $U(n)$ on finite quotients of $\Cb^n$:  therefore we start by reviewing this ansatz, see \cite{FIK03}.

>From Section \ref{cao-spaces} onwards, only the soliton on $\Cb^n$ will be used.  The family of solitons on $\cO(-n)$ is reviewed because it is the motivating example for the solitons on desingularizations of  $(\Cb^n\setminus\{0\})/G$ constructed in the rest of this article (see Remark \ref{cao-family-deformations}).

Let $\Phi:\Rb\to\Rb$, and denote its derivative by $\varphi(t)=\Phi_t(t)$.  Throughout this article we use interchangeably the variables $z\in\Cb^n\setminus\{0\}$, $r\in\Rb^+$, $t\in\Rb$, with
\[
|z|^2=r^2=e^t.
\]
We introduce real symmetric 2-tensors on $\Cb^n\setminus\{0\}$:
\begin{align*}
g_{FS}&=g_{S^{2n-1}}-\eta^2,\\
g_{cyl}&=\frac{dr^2}{r^2} +\eta^2 = \frac{1}{4}dt^2+\eta^2
\end{align*}
where $g_{S^{2n-1}}$ is the round sphere metric, the 1-form $\eta$ is the connection 1-form on the Hopf bundle $S^{2n-1}\rightarrow\CP^{n-1}$, and $g_{FS}$ is the pull-back of the standard Fubini-Study metric of $\CP^{n-1}$ with holomorphic sectional curvature $-4$. We denote by $\omega_{FS}$ and $\omega_{cyl}$ the associated (1,1)-forms.

The Cao $U(n)$-invariant Kähler forms on $\Cb^n\setminus\{0\}$ will be of the form
\begin{equation}
\omega=\tfrac{i}{2}\de\db\Phi(t) =\varphi(t)\omega_{FS}+\varphi_t(t)\omega_{cyl}.\label{eq:5}
\end{equation}
This formula defines a Kähler metric outside the origin if and only if both $\varphi$ and $\varphi_t$ are everywhere positive.  We henceforth assume that they are. The Ricci form is then
\begin{equation}
 \rho(\omega) = i\ddb \log \frac{\omega^n}{\vol_{\Cb^n}} = i\ddb \log \big(\varphi^{n-1}\varphi_te^{-nt}\big) = i \ddb \big(\log (\varphi^{n-1}\varphi_t) - nt\big).\label{eq:1}
\end{equation}

Fix $\mu\in \Rb$. Let $X$ be the radial vector field $-2\mu r\partial/\partial r=-4\mu\partial/\partial t$.  The pair $(\omega,X)$ is a steady Ricci soliton if
\begin{equation*}\label{soliton-general}
\Ric(g)+\tfrac{1}{2}\cL_Xg=0,
\end{equation*}
which by (\ref{eq:1}) can be rewritten
\[ \varphi^{n-1}\varphi_te^{\mu\varphi}=cst. \, e^{nt}. \]
By a translation on $t$ (that is, an homothety of $\Cb^n$), one can take the constant to be 1, and the final equation is therefore
\begin{equation}\label{soliton-eq}
\varphi^{n-1}\varphi_te^{\mu\varphi}= e^{nt}.
\end{equation}

Such Ricci solitons are \emph{gradient} Ricci solitons:  with the gradient function $f(z)=-\mu\varphi(t)$,
\[ 
\Ric(g_\Phi)+(\nabla^{g_\Phi})^2f=0.
\]
By multiplying $\varphi$ by  a constant, one can restrict to the case $\mu=1$, which we will henceforth do.

The equation (\ref{soliton-eq}) is solved in the following way.
Let $F(s)$ be the degree-$(n-1)$ polynomial $F(s)=\sum_{r=0}^{n-1}(-1)^{n-r-1}\frac{(n-1)!}{r!}s^r$. Then one can check that
\begin{equation*}
  \label{eq:2}
  \frac{d}{ds}\left[F(s)e^s\right]=s^{n-1}e^s, 
\end{equation*}
and for all $s> s_0\geq 0$,
\begin{equation}
0<F(s)e^s-F(s_0)e^{s_0}<s^{n-1}e^s.\label{eq:4}
\end{equation}
It follows that the soliton equation (\ref{soliton-eq}) can be rewritten as
\begin{equation*}
  \label{eq:3}
  F(\varphi)e^\varphi = \frac{e^{nt}}n + cst.
\end{equation*}
We can therefore define a family of solitons parametrized by a nonnegative real number $a\geq0$ by taking $\varphi_a=\varphi$ to be implicitly defined by the equation
\begin{equation}\label{defining-phi}
F(\varphi(t))e^{\varphi(t)}=\frac{e^{nt}}n+F(a)e^{a}.
\end{equation}
>From (\ref{eq:4}) the function $\varphi$ is well defined, with $\varphi(t)>a$ for all $t$; again from (\ref{eq:4}), one has $\varphi^{n-1}e^\varphi>\frac{e^{nt}}n$ so equation (\ref{soliton-eq}) implies that
\begin{equation}
  \label{eq:6}
  0 < \varphi_t < n.
\end{equation}
Thus $\Phi:=\int\varphi dt$ is the Kähler potential of a steady Ricci soliton with associated holomorphic vector field $X=-2 r\partial/\partial r=-4\partial/\partial t$.

\begin{example}[Hamilton's cigar soliton]
In dimension $n=1$, the equation (\ref{defining-phi}) reduces to $e^{\varphi_a(t)}= e^t+e^{a}$, and then $(\varphi_a)_t=(1+e^{a}e^{-t})^{-1}$ by (\ref{soliton-eq}). The associated Kähler-Ricci soliton on $\Cb$ is $(\varphi_a)_tg_{cyl}$, so the dependence in $a$ is just by translation, and we get a unique soliton (Hamilton's cigar soliton),
\[ (\varphi_0)_tg_{cyl}=\frac{dx^2+dy^2}{r^2+1}.  \]
\end{example}

\section{Asymptotics of the Cao potentials} \label{asymptotics}

In this section, we establish the asymptotics of the functions $\varphi_a(t)$ constructed in the previous section as $t\to-\infty$ and as $t\to\infty$.  The former is necessary to determine the topological structure of the metric completion of the solitons.  The latter determines the asymptotics at infinity of the Cao solitons.

\begin{lemma} \label{t-neg-infty}
As $t\to-\infty$, 
\begin{equation*}
\varphi_a(t)= \begin{cases}
a+\tfrac{1}{n} a^{1-n}e^{- a}e^{nt}+O(e^{2nt}), & a>0.\\
e^t-\tfrac{1}{n+1}e^{2t}+O(e^{3t}),&a=0.
\end{cases}
\end{equation*}
\end{lemma}

\begin{proof} Easy calculation starting from (\ref{defining-phi}).
\end{proof}

Since $e^t=r^2$, when $a=0$, the metric associated to $\varphi=\varphi_0$ on $\Cb^n\setminus\{0\}$ extends smoothly across the origin ($\varphi_0$ has a smooth development in powers of $e^t$). This is Cao's steady soliton (unique up to translation and rescaling) on $\Cb^n$.

When $a>0$,  then $\varphi_t(t)=a^{1-n}e^{-a}e^{nt}+O(e^{2nt})$ with $e^{nt}=r^{2n}$ (and again there is a smooth development in powers of $e^{nt}$), so formula (\ref{eq:5}) says that the metric associated to $\varphi=\varphi_a$ on the quotient $(\Cb^n\setminus\{0\})/\Zb_n\cong \cO(-n)\setminus \CP^{n-1}$ extends smoothly across the zero section $\CP^{n-1}$ to give a smooth metric on $\cO(-n)$.  This is Cao's family of steady solitons (a one-parameter family, up to translation and rescaling) on $\cO(-n)$.

\begin{remark}\label{cao-family-deformations}
  As $a\rightarrow0$,  the solitons on $\cO(-n)$ converge to the Cao soliton on $\Cb^n/\Zb_n$ with an orbifold singularity at the origin. This is an explicit example of our construction in this paper. It is also interesting to see what `bubble' occurs in this limit: an easy calculation gives \[\frac{\varphi_a(t+\log a)}a \rightarrow (1+e^{nt})^{\frac1n}\] (the translation in $t$ corresponds to homotheties in $\Cb^n$). The limit is the well-known Kähler Ricci-flat metric on the total space of $\cO(-n)$ constructed by Calabi \cite{Cal79} (for $n=2$ this is the Eguchi-Hanson metric).
\end{remark}

We now study the behaviour at infinity. The following lemma again follows directly from equation (\ref{defining-phi}).

\begin{lemma}\label{t-pos-infty}
As $t\to\infty$,
\begin{align*}
\varphi(t)&=nt-(n-1)\log t -n\log n+O\left(\frac{\log t}{t}\right).\\
\varphi_t(t)&=n-\frac{n-1}{t} +O\left(\frac{\log t}{t^2}\right).
\end{align*}\qed
\end{lemma}
Therefore the metric has the asymptotic behaviour when $t\rightarrow+\infty$
\begin{equation}
  \label{eq:7}
  g \sim n \big( \tfrac14 dt^2 + \eta^2 + t g_{FS} \big)
\end{equation}
and in particular the volume of the ball of radius $t$ is of order $t^n$, that is half-dimensional; also the injectivity radius is bounded below. More generally:
\begin{lemma}\label{lem:bounded-curvature}
  The Cao solitons  have injectivity radius bounded below and curvature and all their covariant derivatives bounded.
\end{lemma}
\begin{proof}
  It remains to be proved that the Riemannian curvature and its derivatives are bounded. Given the explicit form of the metric, this  reduces to the fact that all derivatives of $\varphi$ are bounded, which is obvious (at infinity $\varphi_t\rightarrow n$ and the higher derivatives go to zero).
\end{proof}

\section{Analysis on the Cao soliton} \label{cao-spaces}

For the rest of this article, we need only the soliton obtained for $a=0$; that is, the Cao soliton on $\Cb^n$ with Kähler form $\omega_0=\varphi\omega_{FS}+\varphi_t\omega_{cyl}$.  In this section we do some analysis on this soliton and on manifolds asymptotic to it.

We first work generally on a Riemannian manifold $(M,g)$ with a single end isometric to the portion $\{T_0\leq t\}$ of the Cao soliton $(\Cb^n,\omega_0)$, or to its quotient by some freely-acting $G\subseteq SU(n)$.
Fix $\delta>0$ and define a weight function
\begin{equation*}
  \label{eq:8}
  w(t) = e^{\delta \varphi}.
\end{equation*}
So $w\geq1$ and at infinity $w(t)\sim e^{\delta nt}$. We use the Hölder weighted spaces $C^{k,\alpha}_\delta(M)=w^{-1}C^{k,\alpha}(M)$, where $C^{k,\alpha}(M)$ are the usual Hölder spaces.

\begin{proposition}[Schauder estimate for Cao asymptotics]\label{schauder-cao}
    Let $a$ be a symmetric positive-definite bivector field, and $b$ a vector field, whose (unweighted) $C^{\alpha}$ norms are finite. Suppose moreover that for some real $\lambda>0$, the uniform global bound $a\geq\lambda g^{-1}$ holds.  Denote by $L$ the second-order elliptic operator
  \[
Lu=\langle a, \nabla^2u\rangle +\langle b, du\rangle.
\]

Then there exists a constant $C$, such that
    for all $t_0$, $t_1$ with $T_0+1\leq t_0+1\leq t_1<\infty$, for all $u$ on $\{t< t_1\}\subseteq M$,
  \[
  \|u\|_{C^{2,\alpha}_\delta(\{t< t_0\})}
  \leq C\left[\|u\|_{C^{0}_\delta(\{t< t_1\})}+\|Lu\|_{C^{\alpha}_\delta(\{t< t_1\})}\right].
  \]
\end{proposition}

\begin{proof}
  By Lemma \ref{lem:bounded-curvature}, there exist $s>0$ and $Q>0$ such that for all $x\in M$, there exist (harmonic) co-ordinates on the ball $B_x(s)$ in which the metric tensor $g$ is $C^{1,\alpha}$ controlled by $Q$.  See, for example, \cite[Section 1.2]{hebey}.  Therefore on balls of radii $\leq s$, the Schauder norms with respect to the metric $g$ are uniformly equivalent to the Euclidean Schauder norms.

  Moreover, by assumption, the coefficients $a$, $b$ of the operator are uniformly controlled on balls $B_x(s)$.
  Therefore the Euclidean Schauder estimate \cite[Theorem 6.2]{GT} also holds, with uniform constant, on balls $B_x(s)$ with Riemannian Schauder norm.

  The weight function $w_\delta$ has the property that, for some uniform constant $C$, for all $x$ and all $y\in B_x(s)$,
$C^{-1}w(x)\leq w(y)\leq C w(x)$.
  We may therefore combine the uniform Schauder estimates on balls $B_x(s)$ to give the desired global weighted estimate.
\end{proof}

We now specialize to the Cao manifold $(\Cb^n, \omega_0)$ itself, with the weight function
$w(t) = e^{\delta \varphi}$.
itself. 
Denote by $\Delta$ the Laplace-Beltrami operator of $\omega_0$, and by $X$ (as in Section  \ref{radial-potential}) the radial vector field $-2r\partial/\partial r=-4\partial/\partial t$.  The linearization of the soliton equation on a Kähler potential is the operator $\Delta-X$.  We have the following $C^0$ estimates for this operator:

\begin{lemma}Let $0<\delta<1$.
  \begin{enumerate}
  \item For $t_0\in\Rb$, and $u$ on $\{t\leq t_0\}\subseteq \Cb^n$,
        \begin{equation}
\label{eq:12}
        \sup_{t\leq t_0} w |u| \leq \max\left[ \sup_{t= t_0} w |u|, \ \frac1{4\delta(1-\delta)n} \sup_{t\leq t_0} w |(\Delta-X)u| \right].
    \end{equation}
  \item For $r_0\in [1,\infty]$, and $u$ on $\{1\leq r\leq r_0\}\subseteq \Cb^n$ such that $u|_{\{r=r_0\}}=0$ (if $r_0<\infty$) or $u\in C^2_\delta(\{1\leq r\})$ (if $r_0=\infty$),
    \begin{equation}
      \label{eq:28}
        \sup_{1\leq r\leq r_0} w |u| \leq \max\left[ \sup_{r= 1} w |u|, \ \frac1{4\delta(1-\delta)n} \sup_{1\leq r\leq r_0} w |(\Delta-X)u| \right].
    \end{equation}
  \end{enumerate}
\end{lemma}

\begin{proof}
  The Laplacian of a radial function is
  \begin{equation*}
    \label{eq:9}
    \Delta f = \frac{4}{\varphi^{n-1}\varphi_t}
             \frac{\partial}{\partial t}\big( \varphi^{n-1} \frac{\partial f}{\partial t}\big).
       \end{equation*}
Therefore, using equation (\ref{soliton-eq}) and $X=-4\frac{\partial}{\partial t}$,
       \begin{align*}
         (\Delta-X)e^{-\delta\varphi}
         &= -\frac{4\delta}{\varphi^{n-1}\varphi_t} \frac{\partial}{\partial t}\big(e^{-\delta\varphi}\varphi^{n-1}\varphi_t\big) + 4\delta\varphi_t \\
         &= -\frac{4\delta}{e^{nt-\varphi}}\frac{\partial}{\partial t}\big(e^{nt-(1+\delta)\varphi}\big) + 4\delta\varphi_t \\
         &= -4\delta(n-\delta\varphi_t)e^{-\delta\varphi}.
\end{align*}
Since $0<\varphi_t<n$ by (\ref{eq:6}), one deduces
\begin{equation*}
  \label{eq:10}
  4\delta(1-\delta)n < - e^{\delta\varphi}(\Delta-X)e^{-\delta\varphi} < 4\delta n.
\end{equation*}

This estimate enables the use of a barrier argument: if $(\Delta-X)f=g$, then write $f=e^{-\delta\varphi}F$; then
\begin{equation}
  \label{eq:11}
  (\Delta-X)F + F e^{\delta\varphi} (\Delta-X)e^{-\delta\varphi} - 2\delta \langle d\varphi,dF \rangle = e^{\delta\varphi}g.
\end{equation}

On a compact domain $t\leq t_0$,  the maximum principle applied to (\ref{eq:11}) implies
\begin{equation*}
 \sup_{\{t\leq t_0\}} e^{\delta\varphi}|f| = \sup_{\{t\leq t_0\}} |F| \leq \max\left[\sup_{\{t=t_0\}}|F|+\frac1{4\delta(1-\delta)n} \sup_{\{t\leq t_0\}} e^{\delta\varphi}|g|\right].
 \end{equation*}
This is (\ref{eq:12}).  Similarly, if $r_0<\infty$, then on the compact domain $1\leq r\leq r_0$, if $u|_{\{r=r_0\}}=0$, then we obtain (\ref{eq:28}).

Finally, if $r_0=\infty$, we may select a sequence of compactly-supported functions $\chi_k$, with $0\leq \chi_k\leq 1$, with $\chi_k\to 1$ pointwise as $k\to \infty$, such that for all $u\in  C^2_\delta(\{1\leq r\})$,
\[
\lim_{k\to\infty}\|(\Delta-X)(\chi_ku)\|_{C^0_\delta(\{1\leq r\})}=\|(\Delta-X)u\|_{C^0_\delta(\{1\leq r\})}.
\]
(Indeed, we may take $\chi_k(t)=\chi(t/k)$, for some fixed cutoff function $\chi$.)  We can then extract the noncompact ($r_0=\infty$) case of (\ref{eq:28}) from the sequence of compact results.
\end{proof}

\begin{theorem}[Drift Laplacian on the Cao soliton]\label{th:analysis-cao}
  Let $0<\delta<1$. Then the operator \[\Delta-X:C^{2,\alpha}_\delta(\Cb^n)\rightarrow C^\alpha_\delta(\Cb^n)\]
  is an isomorphism.
\end{theorem}
\begin{proof}

  To prove surjectivity: let $g\in C^\alpha_\delta(\Cb^n)$.  For each of a sequence of domains $\{t\leq t_i\}$, with $t_i\to\infty$, let $f_i$ be the solution to the Dirichlet problem $(\Delta-X)f_i=g$, $f_i|_{\{t=t_i\}}=0$.  Such a solution exists since $\Delta-X$ has no zero-th order term.   By the estimate (\ref{eq:12}),
  \[
  \sup_{t\leq t_i} w |f_i| \leq \frac1{4\delta(1-\delta)n} \sup_{t\leq t_i} w |(\Delta-X)f_i|=\frac1{4\delta(1-\delta)n} \sup_{t\leq t_i} w|g|.
\]
So, by the Schauder estimate Proposition \ref{schauder-cao}, for all $t_0$ and all $i$ such that $t_i\geq t_0+1$,
\[
\|f_i\|_{C^{2,\alpha}_\delta(\{t\leq t_0\})}
  \leq C\left[\|f_i\|_{C^{0}_\delta(\{t\leq t_i\})}+\|(\Delta-X)f_i\|_{C^{\alpha}_\delta(\{t\leq t_i\})}\right]\leq C \|g\|_{C^{\alpha}_\delta(\{t\leq t_i\})},
  \]
where the constant $C$ is independent both of $t_0$ and of $i$.

By a diagonal argument, we can extract a subsequence, also denoted $f_i$, which $C^2$-converges on each compact subset.  Let $f$ denote the limit.  On each compact subset $\{t\leq t_0\}$, we have that
\[
\|f\|_{C^{2,\alpha}_\delta(\{t\leq t_0\})}
\leq \limsup_{i\to\infty}\|f_i\|_{C^{2,\alpha}_\delta(\{t\leq t_0\})}
  \leq C \|g\|_{C^{\alpha}_\delta}.
\]
So $\|f\|_{C^{2,\alpha}_\delta}\leq c\|g\|_{C^\alpha_\delta}$. This proves the surjectivity of $\Delta-X$.

To prove injectivity: suppose that $(\Delta-X)f=0$ with $f=O(e^{-\delta\varphi})$.  Let $\delta'$ be such that $0<\delta'<\delta$.  Applying the estimate (\ref{eq:12}) with the parameter $\delta'$,
\begin{equation*}
  \sup_{t\leq t_0} e^{\delta'\varphi}|f| \leq \sup_{t=t_0} e^{\delta'\varphi}|f|
  =e^{(\delta'-\delta)\varphi(t_0)}\sup_{t=t_0} e^{\delta\varphi}|f|.
\end{equation*}
The right hand side tends to 0 as $t_0\rightarrow\infty$, so $f=0$.
\end{proof}

\section{Analysis on finite quotients of the Cao soliton}
\label{sec:analys-finite-quot}

For the rest of this article, fix a finite subgroup $G\subset SU(n)$ acting freely on $\Cb^n\setminus\{0\}$.  We now get an orbifold soliton $(C_G=(\Cb^n\setminus\{0\})/G,\omega_0)$ with an isolated singularity at the origin.

It is possible to construct weighted orbifold Hölder spaces on $C_G$, consisting of the functions on $C_G$ whose lifts to $\Cb^n\setminus\{0\}$ extend to functions in the spaces $C^{k,\gamma}_\delta(\Cb_n)$ of the previous section.  By
Theorem \ref{th:analysis-cao}, the operator $\Delta - X$ is an isomorphism between appropriate such weighted-at-infinity orbifold  Hölder spaces.  However, this result is not sufficient for our gluing procedure.

We need to consider Hölder spaces which are also nontrivially weighted  at the origin, with weight functions polynomial in the distance to the origin.  We take as the weight function the positive radial function
\[
w_0(t):=\begin{cases}e^{\gamma t/2},&t<0, \\
e^{\delta\varphi(t)},&t>0,
\end{cases}
\]
(we can smooth this function at $t=0$ but this is not really needed). So $w_0(t)\sim e^{\delta nt}$ at infinity (the weights of the previous section) and $w_0(t)\sim r^\gamma$ at the origin. If it is necessary to specify the parameters, we will use the notation $w_{0;\gamma,\delta}$ for $w_0$.

Choose a positive function $\sigma:C_G\rightarrow\Rb$ such that at each point $x$, the injectivity radius at $x$ is greater than $\sigma(x)$. We can take $\sigma=\sigma(t)$ to be a radial function, with
\begin{equation}\label{sigma-def}
  \sigma(t) = \begin{cases} cst. \, e^{t/2} = cst. \, r, & t \ll 0, \\
    cst, & t \gg 0. \end{cases}
  \end{equation}
We can now define the weighted Hölder space $C^{k,\alpha}_{\gamma,\delta}(C_G)$ by the norm
\begin{equation}
 \|f\|_{C^{k,\alpha}_{\gamma,\delta}} =
  \sum_{i=0}^k \sup \sigma(x)^i w_0(x) |\nabla^i f| + \sup \sigma(x)^{k+\alpha} w_0(x) \sup_{y\in B_x(\sigma(x))} \frac{|P_{y,x}\nabla^k f(y) - \nabla^kf(x)|}{d(x,y)^\alpha},\label{eq:15}
\end{equation}
where $P_{y,x}$ is the parallel transport along the unique minimizing geodesic from $y$ to $x$.  (We use the Riemannian metric $\omega_0$ for all geometric aspects of this construction:  $\nabla$, $|\cdot|$, $d(\cdot,\cdot)$, $P_{x,y}$.) This weighted Hölder norm is equivalent to the previous section's norm  on $\Cb^n$, for functions supported on $\{t>0\}$, and to the norm classically used on cones for functions supported on $\{t<0\}$.

Before we do this, we note, for future reference, some (essentially classical) facts  about $\Delta-X$ when considered as an operator on a bounded portion $\{t\leq t_0\}$ of $C_G$. Let $C^{k,\alpha}_\gamma(t\leq t_0)$ denote the H\"older space of functions on that portion, with norm the analogue of the formula (\ref{eq:15}) restricting to $x,y\in\{t\leq t_0\}$.    Then let   $\hat C^{2,\alpha}_\gamma(t\leq t_0)$ denotes the  set of functions  $f\in C^{2,\alpha}_\gamma(t\leq t_0)$ with (Dirichlet) boundary condition $f|_{t=t_0}=0$.

 \begin{lemma}\label{eq:14}
For each $\gamma\in(0,2n-2)$, the operator
   \[
   \Delta-X : \hat C^{2,\alpha}_\gamma(t\leq t_0) \longrightarrow C^\alpha_{\gamma+2}(t\leq t_0)\] is an isomorphism.
 \end{lemma}

 \begin{proof}
   We note that the only critical weights of $\Delta$ at the origin in the interval $[0,2n-2]$ are  $0$ (corresponding to the constants) and $2n-2$ (corresponding to the Green function).  
   This implies that the Laplacian $\Delta$ has index zero as a map between the spaces $\hat C^{2,\alpha}_\gamma(t\leq t_0)$ and $C^\alpha_{\gamma+2}(t\leq t_0)$ (its adjoint is itself with the adjoint weight $\gamma'=2n-2-\gamma\in (0,2n-2)$, so the operator and its adjoint have the same kernel).

Since      $X=-2r\frac \partial{\partial r}$ is a compact operator between the spaces in question, this means that $\Delta-X$ has index zero.

 To show the injectivity of the operator, observe that the calculation of the critical weights implies that a solution of $(\Delta-X)f=0$ with $f\in \hat C^{2,\alpha}_{\gamma}(t\leq t_0)$ and $0<\gamma<2n-2$ must actually be bounded and have a limit at the origin. By local elliptic regularity it must  be smooth (in the orbifold sense of lifting to a smooth function on the $G$-prequotient $B=B_0(\exp(t_0/2))\subseteq\Cb^n$). So we have a solution of $(\Delta-X)f=0$ in the space $C^{2,\alpha}(B)$, with Dirichlet boundary conditions, so $f=0$.

 Therefore the operator is an isomorphism.
\end{proof}

 We can now give the doubly-weighted isomorphism.
 
\begin{theorem}[Drift Laplacian on finite quotients of the Cao soliton]\label{th:isoCao}
  Let $0<\delta<1$ and $0<\gamma<2n-2$.  Then the operator \[\Delta-X:C^{2,\alpha}_{\gamma,\delta}(C_G)\rightarrow C^\alpha_{\gamma+2,\delta}(C_G)\]
   is an isomorphism.
\end{theorem}
\begin{proof}
  This theorem will be a consequence of Theorem \ref{th:analysis-cao} (which deals with the analysis at infinity) and of the local theory for polynomial weights in cone points (as in Lemma \ref{eq:14}).

 To show the injectivity of the operator, we follow the same argument as in Lemma \ref{eq:14}.  A solution of $(\Delta-X)f=0$ with $f\in C^{2,\alpha}_{\gamma,\delta}$ and $0<\gamma<2n-2$ must actually be bounded and have a limit at the origin, and so lift to a solution of $(\Delta-X)f=0$ in the space $C^{2,\alpha}_\delta$. We then find $f=0$ by Theorem \ref{th:analysis-cao}.

To show the surjectivity of the operator, take $g\in C^\alpha_{\gamma+2}(C_G)$. By Lemma \ref{eq:14} we find $f_0\in C^{2,\alpha}_\gamma(\{t\leq0\})$ defined for $t\leq 0$, such that $(\Delta-X)f_0=g$. Take $\chi$ a cutoff function so that $\chi(t)=0$ for $t\geq0$ and $\chi(t)=1$ for $t\leq-1$. Then, using the isomorphism of Theorem \ref{th:analysis-cao}, and the fact that $g - (\Delta-X)(\chi f_0)$ has support included in $\{t\geq-1\}$, we can find $f_1\in C^{2,\alpha}_\delta(C_G)\subseteq C^{2,\alpha}_{\gamma,\delta}(C_G)$ such that
 \[ (\Delta-X)f_1 = g - (\Delta-X)(\chi f_0). \]
 Then $f=\chi f_0 + f_1\in  C^{2,\alpha}_{\gamma,\delta}(C_G)$ is the required solution to $(\Delta-X)f=g$.
\end{proof}

We also note the appropriate Schauder estimate on these doubly-weighted spaces:
\begin{proposition}[Schauder estimate on the Cao soliton quotient]\label{schauder-cao-quotient}
  There exists a constant $C$, such that
    for all $t_0\leq\infty$, for all $u$ on $\{t\leq t_0\}$, with also $u|_{\{t=t_0\}}=0$ if $t_0<\infty$,
  \[
  \|u\|_{C^{2,\alpha}_{\gamma,\delta}(\{t< t_0\})}
  \leq C\left[\|u\|_{C^{0}_{\gamma,\delta}(\{t< t_0\})}+\|(\Delta-X)u\|_{C^{\alpha}_{\gamma+2,\delta}(\{t< t_0\})}\right].
  \]
\end{proposition}

\begin{proof}
  The proof is similar to that of Proposition \ref{schauder-cao}.  In this situation we have that there exists $Q>0$ such that for all $x\in M$, there exist (harmonic) co-ordinates on the ball $B_x(\sigma(x))$ in which the metric tensor $g$ is $C^{1,\alpha}$ controlled by $Q$:  the varying radius $\sigma(x)$ (defined in (\ref{sigma-def})) is needed because the injectivity radius is small near the orbifold point at the origin.

Therefore on balls of radii $\sigma(x)$, the \emph{weighted} Schauder norms with respect to the metric $g$ are uniformly equivalent to the \emph{weighted} Euclidean Schauder norms.
So the weighted Euclidean Schauder \emph{boundary} estimate \cite[Lemma 6.4]{GT} also holds, with uniform constant, on balls $B_x(\sigma(x))$ with weighted Riemannian Schauder norm.

Again, the weight functions $w_{\gamma,\delta}$ have the property that, for some uniform constant $C$, for all $x$ and all $y\in B_x(\sigma(x))$,
$C^{-1}w(x)\leq w(y)\leq C w(x)$.
  We may therefore combine the uniform Schauder estimates on balls $B_x(\sigma(x))$ to give the desired global weighted estimate; the weight $w_{\gamma+2,\delta}$ appears as the product $\sigma^2w_{\gamma,\delta}$.
\end{proof}

\section{Joyce metrics on crepant resolutions} \label{joyce-spaces}

We continue to fix a finite subgroup $G$ of $SU(n)$ acting freely on $\Cb^n\setminus\{0\}$, as in Section \ref{sec:analys-finite-quot}.  For the rest of this article we also fix an equivariant crepant resolution $\pi:J_G\to \Cb^n/G$ of the quotient. In this context, \emph{equivariant} means that the $\Cb^*$ action on $\Cb^n/G$ extends to an action on $J_G$.  Not all such quotients $\Cb^n/G$ admit equivariant crepant resolutions, or even crepant resolutions.  We note some examples which do:

\begin{example}
  In dimension $n=2$, finite subgroups $G\subseteq SU(2)$ which act freely are classified by Dynkin diagrams, falling into two infinite classes $A_k$, $D_k$ together with some exceptional examples.  The associated orbifolds $\Cb^2/G$ are the \emph{Kleinian} or \emph{Du Val singularities}, and there is a unique \emph{minimal} resolution  of each such singularity, which is equivariant and crepant.
\end{example}

\begin{example}
In dimension $n=3$, there are some ten infinite classes of such subgroups $G$. There exists a canonical equivariant crepant resolution of $\Cb^3/G$, the \emph{$G$-orbit Hilbert scheme} of $\mathbb{C}^3$, whose existence was  conjectured  by Nakamura \cite{Nak01} and proved by Bridgeland-King-Reid \cite{BKR01}. There may also exist other crepant resolutions (perhaps not equvariant).
\end{example}

\begin{example}
Toric geometry \cite{DHZ01,CR04} provides numerous examples, in arbitrary dimension, of orbifolds with toric, and therefore equivariant, crepant resolutions.
\end{example}

The equivariance property means that the radial vector field $X=-2R\frac\partial{\partial R}$ on $\Cb^n/G$ extends to $J_G$, and we will continue to denote this extension by $X$; the vector field $JX$ generates the action of $S^1\subset\Cb^*$.

We will use the notation $R$ for the radius function on $J_G$, to distinguish it from the homothetic radius $r$ on the Cao soliton.

As previously for the Cao soliton, we introduce Hölder weighted spaces adapted to the geometry of the manifolds $J_G$. Consider a positive locally bounded weight $w^-_\gamma$, which for $R\geq 2$ satisfies
\[ w^-_\gamma =  R^\gamma. \]
If there is no ambiguity, we will denote $w^-_\gamma$ by $w^-$.

Take any reference metric on $J_G$ which coincides with the flat metric of $\Cb^n/G$ near infinity. There is a positive function $\sigma$ such that at each point $\sigma(x)$ is smaller than the injectivity radius at $x$, and near infinity $\sigma(x)=cst. \, R$. So $w^-_\gamma$ and $\sigma^\gamma$ are comparable ($w^-_\gamma/\sigma^\gamma$ and $\sigma^\gamma/w^-_\gamma$ are bounded). Then the analogue of formula (\ref{eq:15}), with this reference metric, this $\sigma$ and this weight $w^-_\gamma$, defines a weighted Hölder space $C^{k,\alpha}_\gamma(J_G)$.

We recall the Ricci-flat asymptotically locally Euclidean (ALE) Kähler metrics on $J_G$, implicit in \cite{TY91} and constructed directly in  \cite{Joy00}. These metrics are asymptotic to the flat metric on $\Cb^n/G$, up to a term in the weighted space $C^\infty_{2n}(J_G)$. 
\begin{theorem}[Joyce]
  Let $\mathfrak{k}\in H^2(J_G,\Rb)$ be a Kähler class on $J_G$ containing an ALE Kähler metric. Then $\mathfrak{k}$ contains a unique ALE Ricci-flat metric $\omega^-$. Moreover, near infinity one has $\omega^-=\frac i2\ddb \Phi^-$, with
  \[ \Phi^- = R^2 + A R^{2-2n} + \psi, \quad \psi\in C^\infty_\gamma(J_G) \text{ for all }\gamma<2n-1. \]
Finally, if $J_G$ is an equivariant resolution, the metric $\omega^-$ is $S^1$-invariant.
\end{theorem}
The last statement follows from the uniqueness, since the $S^1$ action preserves the ALE behaviour at infinity.

Standard analysis on weighted spaces gives immediately:
\begin{proposition}\label{invert-joyce}
  For any $\gamma\in(0,2n-2)$, the Laplacian is an isomorphism
  \[\Delta:C^{2,\alpha}_\gamma(J_G)\rightarrow C^\alpha_{\gamma+2}(J_G).\]
\end{proposition}

\begin{proposition}[Schauder estimate on the Joyce manifold]\label{schauder-joyce}
  There exists a constant $C$, such that
     for all $u$ on $J_G$,
  \[
  \|u\|_{C^{2,\alpha}_{\gamma}(J_G)}
  \leq C\left[\|u\|_{C^{0}_{\gamma}(J_G)}+\|\Delta u\|_{C^{\alpha}_{\gamma+2}(J_G)}\right].
  \]
\end{proposition}

Since $H^1(J_G,\Rb)=0$, the $S^1$-action generated by $JX$ has a moment map $\mu^- $, that is a function such that
\begin{equation}
 d\mu^- = \iota_{JX}\omega^-.\label{eq:17}
\end{equation}
It follows that
\begin{equation*}
\cL_X\omega^-=d\iota_X\omega^-=-dJd\mu=-2i\ddb \mu^-.\label{eq:16}
\end{equation*}
Since near infinity $\omega^-=\tfrac{i}{2}\ddb \Phi^-$, we observe that
$d\mu^-=\iota_{JX}\tfrac{i}{2}\ddb\Phi^-=-\tfrac{1}{4}d\Phi^-$.
Therefore, up to a constant,
\begin{equation*}
  \label{eq:18}
  \mu^- = - \frac14 X\Phi^- .
\end{equation*}

\section{Generalities on steady K\"ahler-Ricci solitons}\label{uniqueness-solitons}

In this section, we digress to prove some general facts about steady K\"ahler-Ricci solitons, culminating in the uniqueness theorem stated in the Introduction.  We work on an arbitrary noncompact complex manifold $M$.

\begin{lemma}\label{steady-soliton-potential-equations}
  \begin{enumerate}
    \item Suppose that the Kähler forms $\omega$ and $\omega+\frac i2\ddb\Psi$ are both steady solitons with drift vector field $X$.  Then
  \[
  0=i\ddb\left[-\log\frac{(\omega+\frac i2\ddb\Psi)^n}{\omega^n} +\frac14 X\Psi\right].
  \]
\item  Suppose that  $\omega$ is K\"ahler,  that  $\mu$ is such that $\tfrac{1}{2}\cL_X\omega=-i\ddb\mu$, and that $\Omega$ is a global holomorphic $n$-form. Let  $f=-\log \frac{\omega^n}{\Omega\wedge\overline\Omega}-\mu$. Then $\omega+\frac i2\ddb\Psi$ is a steady soliton with drift vector field $X$, if and only if
  \[
  0=i\ddb \left[ -\log \frac{(\omega+\frac i2\ddb \Psi)^n}{\omega^n} + \frac14 X\Psi + f \right].
  \]
 \end{enumerate}
\end{lemma}

\begin{proof}
  For the first part, we calculate
\begin{align*}
0&=  \Ric(\omega+\tfrac i2\ddb \Psi) + \frac12 \cL_X(\omega+\tfrac i2\ddb \Psi)-  \Ric(\omega) - \frac12 \cL_X(\omega)\\
&= -i\ddb \log \frac{(\omega+\frac i2\ddb \Psi)^n}{\omega^n} +i\ddb(\frac14 X\Psi).
\end{align*}
  For the second, 
\begin{align}
  \Ric(\omega+\tfrac i2\ddb \Psi) + \frac12 \cL_X(\omega+\tfrac i2\ddb \Psi)
  &= -i\ddb \log \frac{(\omega+\frac i2\ddb \Psi)^n}{\Omega\wedge\overline\Omega} -i\ddb(\mu-\frac14 X\Psi) \notag \\
  &= i\ddb \left[ -\log \frac{(\omega+\frac i2\ddb \Psi)^n}{\omega^n} + \frac14 X\Psi + f \right].\label{eq:35}
\end{align}
  
\end{proof}

We can now prove the uniqueness of the Kähler-Ricci soliton in its Kähler class, as stated in the Introduction. This relies on the well-known strong maximum principle of Hopf \cite[Theorem 3.5]{GT}.

\begin{proof}[Proof of Proposition \ref{prop:uniqueness}]
  Suppose that the Kähler forms $\omega$ and $\omega+\frac i2\ddb\Psi$ are both steady solitons, with $\Psi$, $X\Psi$ and $\ddb\Psi$ going to 0 at infinity (the last in comparison with $\omega$).  Then 
  \[
  \log\frac{(\omega+\frac i2\ddb\Psi)^n}{\omega^n} -\frac14 X\Psi
  \]
  tends to 0 at infinity, and  by Lemma \ref{steady-soliton-potential-equations}
  \[
  \Delta_\omega\left[  -\log\frac{(\omega+\frac i2\ddb\Psi)^n}{\omega^n} +\frac14 X\Psi\right]=0.\]
By  the  strong maximum principle,  $\Delta_\omega$-harmonic functions on $M$ either do not attain their extrema, or are constant.
 So, since
the function  tends to 0 at infinity, it vanishes.

  By the convexity of the logarithmic function,
\[ 0=\log \frac{(\omega+\frac i2\ddb\Psi)^n}{\omega^n} - \frac14 X\Psi \leq \frac 14(\Delta_\omega-X)\Psi. \]
Again by the strong maximum principle, $(\Delta_\omega-X)$-subharmonic functions either do not attain their maximum, or are constant.  Since $\Psi\to 0$ at infinity, $\Psi\leq 0$ everywhere.

Inverting the roles of $\omega$ and $\omega+\frac i2\ddb\Psi$ gives the other inequality $\Psi\geq0$, so $\Psi=0$.
\end{proof}

\section{The approximate solution and the Monge-Ampère equation}
\label{sec:appr-solut-monge}
We now begin the gluing construction.
For each small $\epsilon>0$, define a diffeomorphism $p_\epsilon(z)=\epsilon^{-1}z$ of $\Cb^n/G$, which also extends as a diffeomorphism of the equivariant resolution $J_G$. Despite the fact that it is the same complex manifold, we wish to distinguish the desingularization of the Cao manifold, denoted $M_G$, from the ALE space $J_G$ used for this desingularization. In this way, we see $p_\epsilon$ as a `blowup map'
\[ p_\epsilon : M_G \longrightarrow J_G . \]
We will always consider $J_G$ equipped with the Joyce metric $\omega^-$ (from Section \ref{joyce-spaces}), whereas we want to construct on $M_G$ a Kähler-Ricci soliton desingularizing the Cao soliton $\omega_0$ on $C_G=(\Cb^n\setminus\{0\})/G$ (from Section \ref{asymptotics}).  Let $\Phi_0$ (on $C_G$) and $\Phi_-$ (on the complement of a compact subset of $J_G$) be as in those sections, so that $\omega_0=\frac i2\ddb \Phi_0$ and $\omega_-=\frac i2\ddb \Phi_-$.

For each small $\epsilon>0$, we define a reference Kähler metric $\omega_\epsilon$ on $M_G$ in the following way. Let $\chi$ be a cutoff function such that $\chi(r)=0$ for $r\leq1$ and $\chi(r)=1$ for $r\geq2$, and $\chi_\lambda(r)=\chi(r/\lambda)$. The gluing procedure will occur around the radius
\[ r_\epsilon = \epsilon^{\frac n{n+1}}. \]
We construct $\omega_\epsilon$ in the following way:
\begin{itemize}
\item on $r\geq 2 r_\epsilon$, let $\omega_\epsilon=\omega_0$;
\item on $r\leq r_\epsilon$, let $\omega_\epsilon=\epsilon^2 p_\epsilon^*\omega_-$;
\item on $r_\epsilon \leq r \leq 2r_\epsilon$, let $\omega_\epsilon=\frac i2 \ddb\Phi_\epsilon$, where
  \begin{equation}
   \Phi_\epsilon= \chi_{r_\epsilon}(r) \Phi_0 + (1-\chi_{r_\epsilon}(r)) \epsilon^2 p_\epsilon^*\Phi^- . \label{eq:23}
    \end{equation}
\end{itemize}
This is well-defined since in the intermediate region, $\omega_\epsilon$ coincides with $\omega_0$ near $r=2r_\epsilon$ and with $\epsilon^2 p_\epsilon^*\omega_-$ near $r=r_\epsilon$.  It is of course $S^1$-invariant.  Observe that the function $\Phi_\epsilon$ defined in \eqref{eq:23} is actually naturally defined on a larger region,  say $r\geq \epsilon$ (corresponding to $R\geq1$ in $J_G$), and that $\Phi_\epsilon$ is a potential for $\omega_\epsilon$ on this larger region.

To see that the forms $\omega_\epsilon$ are indeed K\"ahler for $\epsilon$ sufficiently small, note that we have $\Phi_0-r^2=O(r^4)$ and $\Phi^-(R)-R^2=O(R^{2-2n})$, with control on the derivatives: $\nabla^k(\Phi^-(R)-R^2)=O(R^{2-2n-k})$. It follows that, for $r_\epsilon \leq r \leq 2r_\epsilon$, one has
\begin{align*}
 r^2 - \epsilon^2 p_\epsilon^*\Phi^-  & = O( \epsilon^{2n} r^{2-2n}) = O( r_\epsilon^4 ), \\
 \nabla^2( r^2 - \epsilon^2 p_\epsilon^*\Phi^- ) & = O( \epsilon^{2n} r^{-2n}) = O( r_\epsilon^2 ).
\end{align*}
Therefore on $r_\epsilon \leq r \leq 2r_\epsilon$ we have
\begin{equation*}
 \omega_\epsilon - \frac i2 \ddb r^2 = O( r_\epsilon^2 ),\label{eq:21}
\end{equation*}
so for $\epsilon$ small enough $\omega_\epsilon$ is positive, so is a Kähler form everywhere.  

Since we also have bounds on all derivatives of $\Phi^-$, we more generally get on the same transition region $r_\epsilon\leq r\leq 2r_\epsilon$ the estimates
\begin{equation}
  \label{eq:22}
  \nabla^k(\Phi_\epsilon-r^2) = O(r_\epsilon^{4-k}), \quad \nabla^k(\omega_\epsilon - \frac i2 \ddb r^2) = O( r_\epsilon^{2-k} ).
\end{equation}

As in Section \ref{joyce-spaces}, there is a global moment map $\mu_\epsilon:M_G\rightarrow\Rb$, such that $d\mu_\epsilon=\iota_{JX}\omega_\epsilon$. We can fix the constant by deciding that, for $r\geq \epsilon$ one has
\begin{equation}\label{mu-epsilon}
  \mu_\epsilon = - \frac14 X \Phi_\epsilon
=\begin{cases} - \frac14 X \Phi_0,&  r\geq2r_\epsilon;\\
 - \frac14 X\left[\epsilon^2p_\epsilon^*\Phi^-\right]=\epsilon^2p_\epsilon^*\mu^-,& \epsilon\leq r\leq r_\epsilon.
  \end{cases}\end{equation}
The latter implies that, in fact, $\mu_\epsilon=\epsilon^2p_\epsilon^*\mu^-$ throughout the region $r\leq r_\epsilon$.  One again has  $\cL_X\omega_\epsilon = -2i \ddb \mu_\epsilon$.

Let $\Omega$ be the global holomorphic $n$-form on $M_G$ which agrees, on $C_G=\Cb^n/G$, with the Euclidean $n$-form.  (Such an $n$-form exists since the resolution $\pi:M_G\to C_G$ is crepant.)  Let
\begin{equation}
f_\epsilon=-\log \frac{\omega_\epsilon^n}{\Omega\wedge\overline\Omega}-\mu_\epsilon.\label{eq:20}
\end{equation}
and observe that by (\ref{soliton-eq}), $f_\epsilon$ vanishes on $r\geq 2r_\epsilon$.  Motivated by Lemma \ref{steady-soliton-potential-equations}, we now introduce the operator
\begin{equation}
  \label{eq:19}
  T_\epsilon(\Psi) := \frac{(\omega_\epsilon+\frac i2\ddb \Psi)^n}{\omega_\epsilon^n} - e^{\frac14 X\Psi+f_\epsilon}.
\end{equation}
For the rest of the article we study the operators $T_\epsilon$.

\section{Glued function spaces} \label{glued-spaces}

We define weight functions $w_\epsilon=w_{\epsilon,\gamma,\delta}$ depending on the parameters $\gamma\in(0,2n-2)$ and $\delta\in(0,1)$ (we will add the dependence in the parameters, for example $w_{\epsilon,\gamma}$ to specify the dependence in $\gamma$, only in case of ambiguity) in the following way:
\begin{itemize}
\item for $r\geq r_\epsilon$ the weight $w_\epsilon$ coincides with the weight defined on the Cao soliton: $w_\epsilon=w_0$; recall that $w_0(r)=r^\gamma$ near $r=0$;
\item for $r\leq 2r_\epsilon$ the weight $w_\epsilon$ coincides up to a constant with the weight for the Joyce metric: $w_\epsilon=\epsilon^\gamma p_\epsilon^*w^-$; since for $R$ large one has $w^-(R)=R^\gamma$, this choices ensures that the two definitions coincide for $r_\epsilon\leq r\leq 2r_\epsilon$.
\end{itemize}

We define a weighted Hölder space $C^{k,\alpha}_{\epsilon,\gamma,\delta}(M_G)$,  ``gluing the Hölder spaces of the Cao manifold and of the Joyce manifold,'' by the analogue of formula (\ref{eq:15}) for the glued metrics $\omega_\epsilon$, glued weights $w_{\epsilon,\gamma,\delta}$, and similarly glued $\sigma_\epsilon$.  For example, the $C^0$ norm is
\[ \|f\|_{C^0_{\epsilon,\gamma,\delta}} = \sup w_{\epsilon,\gamma,\delta}|f| . \]
A perhaps-clarifying remark is that, for cutoff functions $\chi_\lambda$ as in the previous section,
\[ \|f\|_{C^{k,\alpha}_{\epsilon,\gamma,\delta}} \sim
  \| \chi_{r_\epsilon}(r)f\|_{C^{k,\alpha}_{\gamma,\delta}(C_G)} + \epsilon^\gamma \| (p_\epsilon)_* (1-\chi_{r_\epsilon}(r))f\|_{C^{k,\alpha}_\gamma(J_G)} .
\]
(The factor $\epsilon^\gamma$ is exactly the factor required to make the two norms be equivalent on the transition region $r_\epsilon\leq r\leq 2r_\epsilon$.)

\begin{lemma}\label{lem:bounded}
  For each $\gamma$, $\delta$, and each $\epsilon>0$, the operator \[\Delta_\epsilon-X:C^{2,\alpha}_{\epsilon,\gamma,\delta}(M_G)\rightarrow C^\alpha_{\epsilon,\gamma+2,\delta}(M_G)\] is continuous, with norm bounded independently of $\epsilon$.
\end{lemma}
\begin{proof}
  We have that $(\Delta_\epsilon-X)u=\langle \omega_\epsilon^{-1}, \nabla^2u\rangle+\langle -X, du\rangle$, where $\langle \cdot,\cdot\rangle$ denotes a full trace.  Since both $\omega_{\epsilon}^{-1}$ and $X$ are bounded with respect to the metric $\omega_\epsilon$, the operator is continuous.
  The function spaces were constructed so that the statement is clear for the Laplacian $\Delta:C^{2,\alpha}_{\epsilon,\gamma,\delta}\rightarrow C^\alpha_{\epsilon,\gamma+2,\delta}$. There remains to understand the action of the vector field $X=-2r\frac\partial{\partial r}$. A priori the radial vector field $R \frac\partial{\partial R}$ on $J_G$ does not act on the weighted spaces, which are constructed for a vector field like $\frac\partial{R \partial R}$. But thanks to the scaling built in the definition of the function spaces, it will act: when we transport the function spaces to $(J_G,\omega^-)$, the statement to prove on the region $r\leq r_\epsilon$ (that is $R\leq \epsilon^{-1}r_\epsilon=\epsilon^{\frac{-1}{n+1}}$) is that we have a bound on the operator
  \[ X : \epsilon^{-\gamma} C^{2,\alpha}_\gamma(\omega^-) \longrightarrow \epsilon^{-\gamma-2} C^\alpha_{\gamma+2}(\omega^-) .\]
Since $R^2\leq \epsilon^{-2}r_\epsilon^2$, this operator on the region $R\leq\epsilon^{-1}r_\epsilon$ has actually small norm bounded by $O(r_\epsilon^2)$.
\end{proof}

\begin{proposition}[Uniform Schauder estimate]\label{eq:26}
  There exists a constant $C$, such that 
for  all $\epsilon>0$ sufficiently small, and all $u$ on $M_G$,
 \[
  \| u \|_{C^{2,\alpha}_{\epsilon,\gamma,\delta}(M_G)}
  \leq C\left[ \| u \|_{C^0_{\epsilon,\gamma,\delta}(M_G)} +\| (\Delta_\epsilon-X)u \|_{C^{0,\alpha}_{\epsilon,\gamma+2,\delta}(M_G)} \right].
  \]
\end{proposition}

\begin{proof}
  This is as for Propositions \ref{schauder-cao-quotient} and \ref{schauder-joyce}, with the difference in the last case (the Joyce part) that we have the additional term $Xu$. But we have seen in the proof of Lemma \ref{lem:bounded} that when we restrict to the region $R\leq\epsilon^{-1}r_\epsilon$, then we have $\|Xu\|_{C^\alpha_{\epsilon,\gamma+2,\delta}} \leq c r_\epsilon^2 \|u\|_{C^{2,\alpha}_{\epsilon,\gamma,\delta}}$, so this term does not change the uniform Schauder estimate.
\end{proof}

\section{Error term}

We want to control $f_\epsilon$ defined by \eqref{eq:20}. We know that $f_\epsilon$ vanishes on $r\geq 2r_\epsilon$, so the needed estimate is only on the Joyce manifold. We prove:
\begin{proposition}\label{zeroth-order}
  For all $\gamma\geq-4$, $\delta$, and $k$,  there exists $C$ such that for $\epsilon$ sufficiently small,
  \[  \| f_\epsilon \|_{C^{k}_{\epsilon,\gamma+2,\delta}} \leq C r_\epsilon^{4+\gamma}.  \]
\end{proposition}

\begin{corollary}\label{coro:T0}
  Under the same hypotheses, one has
  \[
  \| T_\epsilon(0) \|_{C^{k}_{\epsilon,\gamma+2,\delta}} = \| 1-e^{f_\epsilon}\|_{C^{k}_{\epsilon,\gamma+2,\delta}} \leq C r_\epsilon^{4+\gamma}.
  \]
\end{corollary}
\begin{proof}[Proof of Corollary \ref{coro:T0}]
  Observe that the bound in Proposition \ref{zeroth-order} reads as
  \[ \epsilon^{\gamma+2} \sup_{R\leq 2\epsilon^{-1}r_\epsilon} w^-_{\gamma+2+k}|\nabla^k f_\epsilon|_{\omega^-} \leq C_k \epsilon^{\frac n{n+1}(4+\gamma)} \]
  that is
  \[ \sup |\nabla^kf_\epsilon|_{\omega^-} \leq C_k ( w^-_{\gamma+2+k})^{-1}\epsilon^{\frac{2n-2-\gamma}{n+1}}\]
  for $\epsilon$ small enough. Expanding $\nabla^k(e^{f_\epsilon}-1)$ in terms of the derivatives of $f_\epsilon$ now gives the same bounds on $e^{f_\epsilon}-1$.
\end{proof}

\begin{proof}[Proof of Proposition \ref{zeroth-order}]
  Since $f_\epsilon$ vanishes for $r\geq 2r_\epsilon$, we only need to bound $f_\epsilon$ for $r\leq 2r_\epsilon$. We will bound independently the two terms $\log \frac{\omega_\epsilon^n}{\Omega\wedge\overline\Omega}$ and $\mu_\epsilon$.

  The first term is nonzero only on the transition region $r_\epsilon\leq r\leq 2r_\epsilon$. From the bounds \eqref{eq:22}, we get that on this region
  \[
  \nabla^k\left( \frac{\omega_\epsilon^n}{\Omega\wedge\overline\Omega}\right)
  =O(r_\epsilon^{2-k}),
  \]
and so
  \[ \sup_{r_\epsilon\leq r\leq 2r_\epsilon} r^{2+\gamma+k} \big| \nabla^k\log \frac{\omega_\epsilon^n}{\Omega\wedge\overline\Omega} \big|
    \leq c_k r_\epsilon^{4+\gamma} \]
  which is exactly the required bound.

 By (\ref{mu-epsilon}), the second term $\mu_\epsilon$ coincides with $-\frac14 X\Psi_\epsilon$ on the region where the potential $\Psi_\epsilon$ exists, that is $r\geq\epsilon$; and with $\epsilon^2p_\epsilon^*\mu^-$ on $r\leq r_\epsilon$. From the estimate \eqref{eq:22} we see that on the transition region $r_\epsilon\leq r\leq2r_\epsilon$ we have
 \[ \nabla^k\mu_\epsilon = O(r_\epsilon^{2-k}), \]
and so 
  \[ \sup_{r_\epsilon\leq r\leq 2r_\epsilon} r^{2+\gamma+k} \big| \nabla^k\mu_\epsilon \big|
  \leq c_k r_\epsilon^{4+\gamma}, \]
  the desired estimate.    Finally, on $r\leq r_\epsilon$ it is easier to write the required estimate in terms of norms on the manifold $J_G$:
  \begin{align*}
\|\mu_\epsilon\|_{C^k_{\epsilon,\gamma+2,\delta}(\{r\leq r_\epsilon\})}
&=\epsilon^{\gamma+2}\|(p_\epsilon^{-1})^*\mu_\epsilon\|_{C^k_{\gamma+2}(\{R\leq \epsilon^{-1}r_\epsilon\})}\\
&=\epsilon^{\gamma+4}\|\mu^-\|_{C^k_{\gamma+2}(\{R\leq \epsilon^{-1}r_\epsilon\})}\\
&\leq Cr_\epsilon^{4+\gamma}
\end{align*}
so long as $\gamma\geq-4$.
\end{proof}

\begin{remark}
  The choice of $r_\epsilon$ was made to get an optimal bound on $f_\epsilon$: the two errors coming from cutting off the Cao soliton, and from gluing the Joyce manifold (which is Ricci-flat, not a soliton), are of the same order around $r_\epsilon$. For example, for $n=2$ we have $r_\epsilon=\epsilon^{\frac23}$. A different choice (such as $r_\epsilon=\epsilon^{\frac12}$, which means taking a bigger $r_\epsilon$) would give a larger error term; it would then be necessary to correct the Joyce metric as in \cite[§~2.4]{BM11} before gluing, but we do not need to do this with our choice of $r_\epsilon$.
\end{remark}

\section{First-order term:  Blowup argument}\label{first-order-sec}

The linearization of the operator $T_\epsilon$ defined in \eqref{eq:19} is $\frac14(\Delta_\epsilon-e^{f_\epsilon}X)$. We first study the simpler operator $\Delta_\epsilon-X$, then, in Corollary \ref{also-valid}, use the control on $e^{f_\epsilon}$ from Corollary \ref{coro:T0} to extend the results to $\Delta_\epsilon-e^{f_\epsilon}X$.

\begin{proposition}\label{first-order-est}
  For each $0<\delta<1$, each $0<\gamma<2n-2$, and each $0<\alpha<1$,  there exists a constant $c$, such that for all $\epsilon$ sufficiently small, for all $r_0\in [1,\infty]$, and for all $u\in C^{2}_{\epsilon,\gamma,\delta}(\{r\leq r_0\})$, with the Dirichlet boundary condition $u|_{r=r_0}=0$ if $r_0<\infty$,
\begin{equation*}
  \| u \|_{C^0_{\epsilon,\gamma,\delta}(r\leq r_0)} \leq c \| (\Delta_\epsilon-X)u \|_{C^0_{\epsilon,\gamma+2,\delta}(r\leq r_0)}.
\end{equation*}
\end{proposition}

\begin{proof}
We prove this by contradiction.  Suppose this false; then there exist a sequence $\epsilon_i\to 0$, a sequence $r_i\in[1,\infty]$, and a sequence of functions $u_i$ defined on $\{r\leq r_i\}$, satisfying $u_i|_{r=r_i}=0$, with
\begin{equation}
\| (\Delta-X)u_i \|_{C^0_{\epsilon_i,\gamma+2,\delta}}\to 0,\label{eq:32}
\end{equation}
but for all $i$,
\begin{equation*}
\| u_i \|_{C^0_{\epsilon_i,\gamma,\delta}}=1.\label{eq:31}
\end{equation*}

    The first observation is that by (\ref{eq:28}), together with the fact that for $r\geq 1$ the  weight $w$ used for the Cao metric coincides with the weights $w_{\epsilon,\gamma}$, $w_{\epsilon,\gamma+2}$ used to define the norms $C^0_{\epsilon,\gamma,\delta}$ and $C^0_{\epsilon,\gamma+2,\delta}$,
\[
        \sup_{1\leq r\leq r_i} w_{\epsilon,\gamma} |u_i| \leq \max\left[ \sup_{r= 1} w_{\epsilon,\gamma} |u_i|, \ \frac1{4\delta(1-\delta)n} \sup_{1\leq r\leq r_i} w_{\epsilon,\gamma+2} |(\Delta-X)u_i| \right].
\]  
Since by hypothesis $\sup w_{\epsilon,\gamma+2} |(\Delta-X)u_i|\rightarrow0$, we see that there exist points $z_i$ such that:
\begin{align*}
  |z_i|&\leq1 \\
  w_{\epsilon,\gamma}(z_i) |u_i(z_i)|&=1.
\end{align*}
We will now use a blowup argument at $z_i$. Up to extracting a subsequence, we can suppose that $r_i\rightarrow r_\infty\in[1,\infty]$.
We distinguish three cases:

{\bf Case 1:}  $0<\inf_{i} |z_i|$.

We can suppose $z_i\rightarrow z_\infty$.
>From local elliptic regularity (including at the boundary $r=r_i$ if it is finite), as in \cite[Theorems 8.32-3]{GT}, we have local $C^{1,\alpha'}$ bounds on $u_i$, so we can extract a subsequence  $C^{1,\alpha}$-converging on each compact subset of  the domain $\{r< r_\infty\}$ of the Cao manifold $C_G$.  Let $u_\infty$ be its limit; since the weights $w_\epsilon$ converge to the weight $w_0$ for $C_G$, we get
\begin{align*}
  (\Delta-X) u_\infty &= 0\text{ (and $u_\infty$ is smooth)}, \\
  \sup w_0 |u_\infty| &= w_0(z_\infty) |u_\infty(z_\infty)| = 1, \\
  u_\infty|_{\{r=r_\infty\}} &= 0 \text{ if }r_\infty<\infty.
\end{align*}
Proposition \ref{schauder-cao-quotient}, the Schauder estimate on the Cao manifold $C_G$, gives that
  \[
  \|u_\infty\|_{C^{2,\alpha}_{\gamma,\delta}(\{r\leq r_\infty\})}
  \leq C\left[\|u_\infty\|_{C^{0}_{\gamma,\delta}(\{r\leq r_\infty\})}+\|(\Delta-X)u_\infty\|_{C^{\alpha}_{\gamma+2,\delta}(\{r\leq r_\infty\})}\right]=C\sup_{\{r\leq r_\infty\}} w_0 |u_\infty|<\infty.
  \]
  So 
  $u_\infty\in C^{2,\alpha}_{\gamma,\delta}(\{r\leq r_\infty\})$ is a nonzero element of the kernel of $\Delta-X$.  If $r_\infty=\infty$, this  contradicts Theorem \ref{th:isoCao}.  If $r_\infty<\infty$, noting that $u_\infty$ also satisfies the Dirichlet boundary condition, this contradicts   Lemma \ref{eq:14}.

{\bf Case 2:}  $\sup_{i} \epsilon_i^{-1} |z_i| < \infty$.

Let $\zeta_i:=p_{\epsilon_i}(z_i)\in J_G$; we can suppose that $\zeta_i\rightarrow\zeta_\infty\in J_G$.

Let $v_i:=\epsilon_i^\gamma (p_{\epsilon_i})_* u_i$. By  definition the weight $w_{\epsilon,\gamma,\delta}$ is equal to the Joyce weight $\epsilon^\gamma p_\epsilon^*w^-_\gamma$ on the larger and larger domains $\{R\leq \epsilon^{-1}r_\epsilon\}$ of $J_G$, so
\begin{equation}
 \sup_{\{R\leq \epsilon^{-1}r_\epsilon\}} w^-_\gamma |v_i| = w^-_\gamma(\zeta_i) |v_i(\zeta_i)| = 1.\label{eq:33}
\end{equation}
On the other hand,
since $\epsilon^{-2}\omega_\epsilon=\omega^-$ for $R\leq \epsilon^{-1}r_\epsilon$, the control \eqref{eq:32} translates into 
\begin{equation}
  \sup_{\{R\leq \epsilon^{-1}r_\epsilon\}} w^-_{\gamma+2} |(\Delta_{\omega^-}-\epsilon^2X) v_i|
  =\sup_{\{r\leq r_\epsilon\}} w_{\epsilon,\gamma+2,\delta} |(\Delta_{\omega_\epsilon}-X) u_i|\longrightarrow 0.\label{eq:34}
\end{equation}
The operators $\Delta_{\omega_-}-\epsilon^2X$ are uniformly locally controlled, so by local elliptic regularity  \cite[Theorems 8.32-3]{GT}, as in Case 1, we can extract a subsequence $C^{1,\alpha}$-converging on each  compact domain of $J_G$.  Let  $v_\infty$ be its limit.  By \eqref{eq:33} and \eqref{eq:34}, 
\[  \Delta_{\omega^-}v_\infty=0, \qquad\sup w_\gamma^- |v_\infty| = 1. \]
We conclude as before, using Proposition \ref{schauder-joyce} (the Schauder estimate on the Joyce manifold) and Proposition \ref{invert-joyce} (the proof of isomorphism).

{\bf Case 3:} We may pass to a subsequence with $z_i\rightarrow0$ and $\epsilon_i^{-1} |z_i|\rightarrow \infty$.

In this case we perform a similar extraction on the cone $\Cb^n/G$ itself, by scaling by $\lambda_i=|z_i|$. As in Case 2, we consider a sequence of homotheties which restrict the points to a compact region.  In this case we take the homotheties $p_{\lambda_i}$, and introduce rescaled functions $v_i:=\lambda_i^\gamma (p_{\lambda_i})_* u_i$. We can suppose $\zeta_i=\frac{z_i}{\lambda_i}\rightarrow\zeta_\infty\in\Cb^n/G$. Finally, now $\lambda_i^{-2}(p_{\lambda_i})_*\omega_{\epsilon_i}$ converges to the flat cone metric on $\Cb^n/G$, and the weights $\lambda_i^{\gamma}p_{\lambda_i}^*w_{\epsilon_i,\gamma,\delta}$ converge to the cone weight $w_\gamma(R)=R^\gamma$, in both instances because $1\gg\lambda_i \gg \epsilon_i$.

The next part of the proof is unchanged from Case 2: the operators
$\Delta_{\lambda_i^{-2}(p_{\lambda_i})_*\omega_{\epsilon_i}}-\lambda_i^2X$
are uniformly locally controlled, and we get a nonzero solution $v_\infty$ of $\Delta v_\infty=0$ on the whole flat cone $\Cb^n/G$, with $\sup_{\Cb^n/G} R^\gamma|v_\infty| < \infty$.

Finally, since $\gamma\in(0,2n-2)$ is not a critical weight of the Laplacian, $\Delta:C^{2,\alpha}_\gamma(\Cb^n/G)\rightarrow C^\alpha_{\gamma+2}(\Cb^n/G)$ is an isomorphism and such a $v_\infty$ cannot exist.
\end{proof}

\begin{proposition} \label{first-order}
  For each $0<\delta<1$, each $0<\gamma<2n-2$, and each $0<\alpha<1$,  there exists a constant $c$, such that for all $\epsilon$ sufficiently small, the operator $\Delta_\epsilon-X:C^{2,\alpha}_{\epsilon,\gamma,\delta}(M_G)\rightarrow C^\alpha_{\epsilon,\gamma+2,\delta}(M_G)$ is an isomorphism, and for all $u\in C^{2,\alpha}_{\epsilon,\gamma,\delta}(M_G)$, 
  \begin{equation}
 \| u \|_{C^{2,\alpha}_{\epsilon,\gamma,\delta}} \leq c \| (\Delta_\epsilon-X)u \|_{C^{0,\alpha}_{\epsilon,\gamma+2,\delta}}.\label{eq:24}
\end{equation}
\end{proposition}
\begin{proof}
First observe that the estimate \eqref{eq:24}, and the injectivity of the operator, follow from the $C^0$ estimate in Proposition \ref{first-order-est}, combined with the uniform Schauder estimate Proposition \ref{eq:26}.

We now show the surjectivity of $\Delta-X$. This is the same argument as in Theorem \ref{th:analysis-cao}.  We want to solve $(\Delta-X)u=v$. On each of a sequnce of bounded domains $\{t\leq t_i\}$, one can solve $(\Delta-X)u_{i}=v$ with $u_{i}|_{r=r_i}=0$, since the operator $\Delta-X$, being a compact perturbation of $\Delta$, has index zero for the Dirichlet problem, and is injective (by 
the maximum principle).

By  Proposition \ref{first-order-est}, together with the Schauder estimate Proposition \ref{schauder-cao}, we have that for all $t_0$ and all $i$ such that $t_i\geq t_0+1$,
\[
\|u_i\|_{C^{2,\alpha}_{\epsilon,\gamma,\delta}(\{t\leq t_0\})}
  \leq  C \|v\|_{C^{\alpha}_{\epsilon,\gamma+2,\delta}(M_G)},
  \]
  where the constant $C$ is independent both of $t_0$ and of $i$.
  Therefore when $i\rightarrow\infty$ we can extract a $C^{2}_{loc}$-limit $u$ still satisfying $\| u \|_{C^{2,\alpha}_{\epsilon,\gamma,\delta}} \leq c \| v \|_{C^0_{\epsilon,\gamma+2,\delta}}$, and solving $(\Delta-X)u=v$.
\end{proof}

\begin{corollary}\label{also-valid}
  Proposition \ref{first-order} is also valid for the family of operators $\Delta_\epsilon-e^{f_\epsilon}X$.
\end{corollary}
\begin{proof}
  It suffices to check that the norm of the operator
  \[ (e^{f_\epsilon}-1)X : C^{2,\alpha}_{\epsilon,\gamma,\delta}(M_G)\longrightarrow  C^\alpha_{\epsilon,\gamma+2,\delta}(M_G) \]
  goes to zero, so that adding it to $\Delta_\epsilon-X$ does not perturb the invertibility and the estimate on the inverse.  This follows easily from Corollary \ref{coro:T0}.
\end{proof}

\section{Second-order term}

Let us write \[ T_\epsilon(\Psi) = T_\epsilon(0) + L_\epsilon(\Psi) + Q_\epsilon(\Psi), \]
where $L_\epsilon=\tfrac{1}{4}(\Delta_\epsilon-e^{f_\epsilon}X)$ is the linearization of $T_\epsilon$ at the origin.
\begin{proposition}\label{Q-bound}
  There exists constants $c,c'>0$ such that if $\|\Psi\|_{C^{2,\alpha}_{\epsilon,\gamma,\delta}}, \|\Psi'\|_{C^{2,\alpha}_{\epsilon,\gamma,\delta}}\leq c \epsilon^{2+\gamma}$, then
  \[ \|Q_\epsilon(\Psi)-Q_\epsilon(\Psi')\|_{C^\alpha_{\epsilon,\gamma+2,\delta}} \leq c' \epsilon^{-2-\gamma} \|\psi-\psi'\|_{C^{2,\alpha}_{\epsilon,\gamma,\delta}} (\|\psi\|_{C^{2,\alpha}_{\epsilon,\gamma,\delta}}+\|\psi'\|_{C^{2,\alpha}_{\epsilon,\gamma,\delta}}) . \]
\end{proposition}
\begin{proof}
  Such an estimate is obvious on a ball, without weights, if the geometry of the metric is uniformly controlled: it comes down to saying that the second derivative $D^2T$ remains bounded on a ball.
 
  We can apply this remark on balls on the Cao part of the manifold ($r\geq1$). The weighted norm on balls is just the usual Hölder norm, multiplied by the factor $w_\epsilon$:
  \[ \|u\|_{C^{k,\alpha}_{\epsilon,\gamma,\delta}(B(x_0,1))} \sim w_\epsilon(x_0) \|u\|_{C^{k,\alpha}(B(x_0,1))}. \]
  Then, for $\|\Psi\|_{C^{2,\alpha}(B(x_0,1))}, \|\Psi'\|_{C^{2,\alpha}(B(x_0,1))}\leq c$ we have the standard estimate
  \[ \|Q_\epsilon(\Psi)-Q_\epsilon(\Psi')\|_{C^\alpha(B(x_0,1))} \leq c' \|\Psi-\Psi'\|_{C^{2,\alpha}(B(x_0,1))} (\|\Psi\|_{C^{2,\alpha}(B(x_0,1))}+\|\Psi'\|_{C^{2,\alpha}(B(x_0,1))}). \]
  Therefore,
  \begin{multline*}
 \|Q_\epsilon(\Psi)-Q_\epsilon(\Psi')\|_{C^\alpha_{\epsilon,\gamma+2,\delta}(B(x_0,1))} \\ \leq c'  w_\epsilon(x_0)^{-1} \|\Psi-\Psi'\|_{C^{2,\alpha}_{\epsilon,\gamma,\delta}(B(x_0,1))} (\|\Psi\|_{C^{2,\alpha}_{\epsilon,\gamma,\delta}(B(x_0,1))}+\|\Psi'\|_{C^{2,\alpha}_{\epsilon,\gamma,\delta}(B(x_0,1))})
\end{multline*}
which gives the required estimate since $w_\epsilon(x_0)\geq1$.

The situation is quite different on the part where we glued the Joyce manifold, because the injectivity radius and the weight become small. First observe that 
\[ T_\epsilon(\Psi) = \frac{(\frac{\omega_\epsilon}{\epsilon^2}+\frac i2\ddb\frac\Psi{\epsilon^2})^n}{(\frac{\omega_\epsilon}{\epsilon^2})^n} - e^{f_\epsilon+\frac{\epsilon^2}4 X\frac \Psi{\epsilon^2}}, \]
and recall that on the Joyce region $r\leq r_\epsilon$, $\frac{\omega_\epsilon}{\epsilon^2}=\omega_-$, the metric on $J_G$, and $f_\epsilon=-\epsilon^2p_\epsilon^*\mu^-$. Therefore on that region
\[ T_\epsilon(\Psi) =
\frac{(\omega_-+\frac i2\ddb\frac\Psi{\epsilon^2})^n}{\omega_-^n}
- e^{\epsilon^2[-\mu^-+\frac{1}4 X\frac \Psi{\epsilon^2}]}. \]
On a large compact part $K\subset J_G$, we have as above, for norms taken with respect to $\omega_-$,
\[ \|Q_\epsilon(\Psi)-Q_\epsilon(\Psi')\|_{C^\alpha} \leq c' \left\lVert\frac \Psi{\epsilon^2}-\frac{\Psi'}{\epsilon^2}\right\rVert_{C^{2,\alpha}} \left(\left\lVert\frac \Psi{\epsilon^2}\right\rVert_{C^{2,\alpha}}+\left\lVert\frac{\Psi'}{\epsilon^2}\right\rVert_{C^{2,\alpha}}\right) \]
for $\|\frac \Psi{\epsilon^2}\|_{C^{2,\alpha}}, \|\frac{\Psi'}{\epsilon^2}\|_{C^{2,\alpha}}\leq c$. But on $K$ we have
\[ \|u\|_{C^{k,\alpha}_{\epsilon,\gamma,\delta}} \sim \epsilon^\gamma \|u\|_{C^{k,\alpha}(\omega_-)}. \]
So we deduce that for $\|\Psi\|_{C^{2,\alpha}_{\epsilon,\gamma,\delta}}, \|\Psi'\|_{C^{2,\alpha}_{\epsilon,\gamma,\delta}}\leq c\epsilon^{2+\gamma}$, we have on $K$ the estimate
  \[ \|Q_\epsilon(\Psi)-Q_\epsilon(\Psi')\|_{C^\alpha_{\epsilon,\gamma+2,\delta}} \leq c' \epsilon^{-2-\gamma} \|\Psi-\Psi'\|_{C^{2,\alpha}_{\epsilon,\gamma,\delta}} (\|\Psi\|_{C^{2,\alpha}_{\epsilon,\gamma,\delta}}+\|\Psi'\|_{C^{2,\alpha}_{\epsilon,\gamma,\delta}}) . \]
  This is again the required estimate.

  The case of the conical region between $K$ and $\{r\geq1\}$ is similar and we will not write the details. The worst constants are obtained on the $K$ part of the manifold, because this is the part of the manifold with smallest injectivity radius and weight. This explains the constants in the statement of the Proposition.
\end{proof}

\section{Inverse function theorem}
\label{sec:inverse-funct-theor}
The proof of Theorem \ref{main} will now follow directly from the following quantitative version of the inverse function theorem, previously used in \cite{BM11}:
\begin{lemma}\label{inverse-function-theorem}
  Let $T:E\to F$ be a smooth map between Banach spaces and define
  $Q:=T-T(0)-DT|_0$.  Suppose that there are positive constants $q$, $r_0$ and $c$, such that
  \begin{enumerate}
  \item $\lVert Q(x)-Q(y)\rVert\leq q\lVert x-y\rVert(\lVert x\rVert+\lVert y\rVert)$ for every $x$ and $y$ in $B_E(0,r_0)$;
  \item $DT|_0$ is an isomorphism with inverse bounded by $c$;
  \item $\lVert T(0)\rVert<\tfrac{1}{2c}\min(r_0,\tfrac{1}{2qc})$.
  \end{enumerate}
  Then the equation $T(x)=0$ admits a unique solution $x$ in $B_E(0,\min(r_0,\tfrac{1}{2qc}))$.
\end{lemma}

\begin{proof}[Proof of Theorem \ref{main}]
  For fixed $0<\delta<1$ and $0<\gamma<2n-2$, consider the family of operators
  \[
T_\epsilon:C_{\epsilon,\gamma,\delta}^{2,\alpha}(M_G)\to C_{\epsilon,\gamma+2,\delta}^{0,\alpha}(M_G).
\]
We will show that for $\epsilon$ sufficiently small, the operator $T_\epsilon$ satisfies the hypotheses of Lemma \ref{inverse-function-theorem}.

By Corollary \ref{coro:T0}, for $\epsilon$ sufficiently small
  \[
  \left\lVert T_\epsilon(0)\right\rVert_{C^{0,\alpha}_{\epsilon,\gamma+2,\delta}(M_G)}
  =\lVert 1-e^{f_\epsilon}\rVert_{C^{0,\alpha}_{\epsilon,\gamma+2,\delta}(M_G)}
  \leq C\epsilon^{(4+\gamma)\frac n{n+1}}.
  \]

  By Corollary \ref{also-valid}, for $\epsilon$ sufficiently small
  \[
L_\epsilon=\tfrac{1}{4}[\Delta_a-  e^{f_a}X ]:C_{\epsilon,\gamma,\delta}^{2,\alpha}(M_G)\to C_{\epsilon,\gamma+2,\delta}^{0,\alpha}(M_G)
  \]
  is an isomorphism, and its inverse is bounded by a constant independent of $\epsilon$.

 By Proposition \ref{Q-bound}, for $\epsilon$ sufficiently small, for all $\Psi$, $\Psi'$ with
  $
  \lVert\Psi\rVert_{C_{\epsilon,\gamma,\delta}^{2,\alpha}},
  \lVert\Psi'\rVert_{C_{\epsilon,\gamma,\delta}^{2,\alpha}}\leq C\epsilon^{\gamma+2}
  $
the map $Q_\epsilon:=T_\epsilon-T_\epsilon(0)-L_\epsilon$ satisfies,
\[
\lVert Q_\epsilon(\Psi)-Q_\epsilon(\Psi')\rVert_{C_{\epsilon,\gamma+2,\delta}^{0,\alpha}}\leq
C\epsilon^{-(\gamma+2)}\lVert\Psi-\Psi'\rVert_{C_{\epsilon,\gamma,\delta}^{2,\alpha}}(\lVert\Psi\rVert_{C_{\epsilon,\gamma,\delta}^{2,\alpha}}+\lVert\Psi'\rVert_{C_{\epsilon,\gamma,\delta}^{2,\alpha}}).
\]
  
Thus if we take $q:=C\epsilon^{-(\gamma+2)}$ and $r_0:=C\epsilon^{\gamma+2}$, then (1) and (2) are satisfied, and moreover
\begin{align*}
  \tfrac{1}{2c}\min(r_0, \tfrac{1}{2qc})&\geq C^{-1}\epsilon^{\gamma+2}\\
  &\geq C^{-1}\epsilon^{-\frac{2n-2-\gamma}{n+1}}\lVert T_a(0)\rVert.
\end{align*}
Since $\gamma<2n-2$, for $\epsilon$ sufficiently small we then have that $ C^{-1}\epsilon^{-\frac{2n-2-\gamma}{n+1}}\geq 1$, and so (3) is satisfied for $\epsilon$ sufficiently small.

Therefore by Lemma \ref{inverse-function-theorem} there exists a solution $\Psi_\epsilon$ to the equation $T_\epsilon(\Psi_\epsilon)=0$, i.e.,
\[
(\omega_\epsilon+\tfrac{i}{2}\ddb \Psi_\epsilon)^n=e^{(X\Psi_\epsilon)/4+f_\epsilon}\Omega\wedge\overline\Omega.
\]
The solution $\Psi_\epsilon$ satisfies, $\left\lVert \Psi_\epsilon\right\rVert_{C^{2,\alpha}_{\epsilon,\gamma,\delta}(M_G)}\leq C\epsilon^{2+\gamma}$.

Since $\Psi\in C_{\epsilon,\gamma,\delta}^{2,\alpha}(M_G)$ is a solution to the equation $T_\epsilon(\Psi)=0$, its differential  $d\Psi$ is a weak solution to a quasilinear elliptic equation
\[
0=dT_\epsilon(\Psi)=\Delta_{\omega_\epsilon+\tfrac{i}{2}\ddb\Psi}d\Psi+F(\nabla^2\Psi,d\Psi).
\]
Bootstrapping using  the Schauder estimate Proposition \ref{schauder-cao}, we conclude that $\Psi\in C_{\epsilon,\gamma,\delta}^{k,\alpha}(M_G)$ for all $k$.  
In particular,  by Lemma \ref{steady-soliton-potential-equations}, $\omega_\epsilon+\frac i2\ddb \Psi$ is a steady K\"ahler-Ricci soliton.
\end{proof}

\bibliography{steady-solitons}{}
\bibliographystyle{alpha}

\sc{Sorbonne Université and École Normale Supérieure}

\sc{Massachusetts Institute of Technology}
\end{document}